\documentclass[10.5pt,leqno]{amsart}
 \usepackage{graphicx}    
 \usepackage{amsmath}
\usepackage{amssymb}
\usepackage[T1]{fontenc}
\usepackage[latin1]{inputenc}
\usepackage[english]{babel}
\usepackage[arrow, matrix, curve]{xy}
\usepackage{amsthm}
\usepackage{amsmath,amscd}
\usepackage{epic,eepic}
\numberwithin{equation}{section}

\DeclareMathOperator{\GL}{GL}

\newcommand{\Ccal}{\mathcal{C}}
\newcommand{\Dcal}{\mathcal{D}}

\newcommand{\Gcal}{\mathcal{G}}

\newcommand{\Ocal}{\mathcal{O}}
\newcommand{\Rcal}{\mathcal{R}}
\newcommand{\Lcal}{\mathcal{L}}

\newcommand{\Fbb}{\mathbb{F}}

\newcommand{\Q}{\mathbb{Q}}

\newcommand{\Z}{\mathbb{Z}}

\newcommand{\Mfrak}{\mathfrak{M}}

\newcommand{\mfrak}{\mathfrak{m}}

\newcommand{\Spec}{{\rm Spec}}
\newcommand{\Spf}{{\rm Spf}}

\newcommand{\id}{{\rm id}}

\newcommand{\Hom}{{\rm Hom}}
\newcommand{\Endo}{{\rm End}}

\newcommand{\Res}{{\rm Res}}

\newcommand{\pr}{{\rm pr}}

\newcommand{\kernel}{{\rm ker}}

\newcommand{\Gal}{{\rm Gal}}
\newcommand{\fl}{{\rm fl}}

\newtheorem{theo}{Theorem}[section]
\newtheorem{lem}[theo]{Lemma}
\newtheorem{prop}[theo]{Proposition}
\newtheorem{cor}[theo]{Corollary}
\theoremstyle{remark}
\newtheorem{rem}[theo]{Remark}
\theoremstyle{remark}

\theoremstyle{definition}

\begin{document}

\title[The image of the coefficient space]{The image of the coefficient space in the universal deformation space of a flat Galois representation of a $p$-adic field}
\author[E. Hellmann]{Eugen Hellmann}

\begin{abstract}
The coefficient space is a kind of resolution of singularities of the universal flat deformation space for a given Galois representation of some local field.
It parametrizes (in some sense) the finite flat models for the Galois representation. The aim of this note is to determine the image of the coefficient space in the universal deformation space.
\end{abstract}
\maketitle

\section{Introduction}
In the theory of deformations of Galois representations one is often interested in a subfunctor of the universal deformation functor consisting of those deformations that satisfy certain extra conditions, so called \emph{deformation conditions} (cf. \cite[§23]{Mazur}). If we deal with a representation of the absolute Galois group of a finite extension $K$ of $\Q_p$ in a finite dimensional vector space in characteristic $p$, there is the deformation condition of being \emph{flat}, which means that there is a finite flat group scheme over the ring of integers of $K$ such that the given Galois representation is isomorphic to the action of the Galois group on the generic fiber. The structure of the ring pro-representing this deformation functor is of interest for modularity lifting theorems (see \cite{Kisin} for example). To get more  information about this structure, Kisin constructs some kind of "resolution of singularities" of the spectrum of the flat deformation ring. This resolution is a scheme parametrizing modules with additional structure that define possible extensions of the representation to a finite flat group scheme over the ring of integers. In \cite{phimod} Pappas and Rapoport globalize Kisin's construction and define a so called \emph{coefficient space} parametrizing all Kisin modules that give rise to the given representation.\\
Following the presentation in \cite{phimod} we want to determine here the image of the coefficient space in the universal deformation space. This question was raised by Pappas and Rapoport in \cite[4.c]{phimod}. Further we show how to recover Kisin's results from the more abstract setting in \cite{phimod}.
The main result of this note is as follows.\\
Let $K$ be a finite extension of $\Q_p$, where $p$ is an odd prime, and $\bar\rho:G_K\longrightarrow \GL_d(\Fbb)$ be a continuous flat representation of the absolute Galois group $G_K=\Gal(\bar K/K)$ on some $d$-dimensional vector space over a finite field $\Fbb$ of characteristic $p$. 
If $\xi:G_K\longrightarrow \GL_d(A)$ is a deformation of $\bar\rho$, we write $C_K(\xi)$ for the coefficient space of (locally free) Kisin modules over $\Spec\ A$ that are related to the flat models for the deformation $\xi$ (see also the definition below).
\begin{theo}
Assume that the flat deformation functor of $\bar\rho$ is pro-representable by a complete local noetherian ring $R^\fl$. We write $\rho:G_K\rightarrow \GL_d(R^\fl)$ for the universal flat deformation. Then the morphism $C_K(\rho)\rightarrow \Spec\ R^\fl$ is topologically surjective.
\end{theo} 
\begin{cor}
If it exists, the flat deformation ring $R^\fl$ is topologically flat, i.e. the generic fiber $\Spec\,R^\fl[1/p]$ is dense in $\Spec\,R^\fl$.
\end{cor}

If the ramification index of the local field $K$ over $\Q_p$ is smaller than $p-1$, then this implies the following result, already contained in \cite{phimod}
\begin{cor}
Denote by $e$ the ramification index of $K$ over $\Q_p$. Assume that the flat deformation functor of $\bar\rho$ is pro-representable and that $e<p-1$. Then $R^\fl$ is the scheme theoretic image of the coefficient space.
\end{cor}

{\bf Acknowledgements:} I want to thank G. Pappas and M. Rapoport for their comments and remarks on a preliminary version of this note. This work was supported by the SFB/TR45 "Periods, Moduli Spaces and Arithmetic of Algebraic Varieties" of the DFG (German Research Foundation). 

\section{Notations}
Let $p$ be an odd prime and $K$ be a finite extension of $\Q_p$ with ring of integers $\Ocal_K$, uniformizer $\pi\in\Ocal_K$ and residue field $k=\Ocal_K/\pi\Ocal_K$. Denote by $K_0$ the maximal unramified extension of $\Q_p$ in $K$ and by $W=W(k)$ its ring of integers, the ring of Witt vectors with coefficients in $k$.
Fix an algebraic closure $\bar K$ of $K$ and denote by $G_K=\Gal(\bar K/K)$ the absolute Galois group of $K$. Further we choose a compatible system $\pi_n$ of $p^n$-th roots of the uniformizer $\pi$ in $\bar K$ and denote by $K_\infty$ the subfield $\bigcup K(\pi_n)$ of $\bar K$. We write $G_{K_\infty}=\Gal(\bar K/K_\infty)$ for its absolute Galois group.\\
Let $d>0$ be an integer and $\Fbb$ a finite field of characteristic $p$. Let $\bar\rho:G_K\rightarrow \GL_d(\Fbb)$ be a continuous representation of $G_K$ and denote by $\bar\rho_\infty=\bar\rho|_{G_{K_\infty}}$ the restriction of $\bar\rho$ to $G_{K_\infty}$.\\
We consider the deformation functors $\Dcal_{\bar\rho}$, $\Dcal^\fl_{\bar\rho}$ and $\Dcal_{\bar\rho_\infty}$ on local Artinian $W(\Fbb)$-algebras with residue field $\Fbb$.
For a local Artinian ring $(A,\mfrak)$ we have
\begin{align*}
\Dcal_{\bar\rho}(A)&=\left\{
{\begin{array}{*{20}c}
\text{equivalence classes of}\ \rho: G_K\rightarrow \GL_d(A)\ \text{such that}\\
\rho\hspace{-2mm}\mod\mfrak\cong \bar\rho
\end{array}}
\right\}
\\
\Dcal_{\bar\rho_\infty}(A)&=\left\{
{\begin{array}{*{20}c}
\text{equivalence classes of}\ \rho: G_{K_\infty}\rightarrow \GL_d(A)\ \text{such that}\\
\rho\hspace{-2mm}\mod\mfrak\cong \bar\rho_\infty
\end{array}}
\right\},
\end{align*}
where two lifts $\rho_1,\,\rho_2$ are said to be equivalent if they are conjugate under some $g\in\kernel(\GL_d(A)\rightarrow \GL_d(A/\mfrak))$.
The functor $\Dcal^\fl_{\bar\rho}$ is the flat deformation functor of Ramakrishna (cf. \cite{Rama}), i.e. the subfunctor of $\Dcal_{\bar\rho}$ consisting of all deformations that are (isomorphic to) the generic fiber of some finite flat group scheme over $\Spec\ \Ocal_K$.
Here "isomorphic to" means isomorphic as $\Z_p[G_K]$-modules, as the action of the coefficients in the generic fiber does not need to extend to the group scheme.\\
If $\Dcal_{\bar\rho}$ (resp. $\Dcal^\fl_{\bar\rho}$) are pro-representable, the pro-representing ring will be denoted by $R$ (resp. $R^\fl$).\\
Recall that $d>0$ denotes an integer and consider the following stacks on $\Z_p$-algebras, defined in \cite{phimod}.
For a $\Z_p$-algebra $R$, write $R_W$ for $R\otimes_{\Z_p}W$ and $R_W((u))$ (resp. $R_W[[u]]$) for $R\widehat{\otimes}_{\Z_p}W((u))$ (resp. $R\widehat{\otimes}_{\Z_p}W[[u]]$), where the competed tensor products are the completions for the $u$-adic topology. Further we denote by $\phi$ the endomorphism of $R_W((u))$ that is the identity on $R$, the Frobenius on $W$ and that maps $u$ to $u^p$.\\
We define an fpqc-stack $\Rcal$ on the category of $\Z_p$-schemes such that for a $\Z_p$-algebra $R$ the groupoid $\Rcal(R)$ is the groupoid of pairs $(M,\Phi)$, where $M$ is an $R_W((u))$-module that is fpqc-locally on $\Spec\ R$ free of rank $d$ as an $R_W((u))$-module, and $\Phi$ is an isomorphism $\phi^\ast M\longrightarrow M$.\\
Further we define a stack $\Ccal$ as follows. The $R$-valued points are pairs $(\Mfrak,\Phi)$, where $\Mfrak$ is a locally free $R_W[[u]]$-module of rank $d$ and $(\Mfrak[1/u],\Phi)\in\Rcal(R)$. For $m\in\Z$ consider the substacks $\Ccal_m\subset \Ccal$ given by pairs $(\Mfrak,\Phi)$ satisfying
\begin{equation}\label{stackCm}
u^m\Mfrak\subset \Phi(\phi^\ast \Mfrak)\subset u^{-m}\Mfrak.
\end{equation}
For $h\in\mathbb{N}$ we write $\Ccal_{h,K}$ for the substack of $\Ccal$ consisting of all $(\Mfrak,\Phi)$ satisfying
\[E(u)^h\Mfrak\subset \Phi(\phi^\ast \Mfrak)\subset \Mfrak.\]
Here $E(u)\in W[u]$ is the minimal polynomial of the uniformizer $\pi\in\Ocal_K$ over $K_0$.\\
In the following we will only consider the case $h=1$ and just write $\Ccal_K$ for $\Ccal_{1,K}$.\\
We will write $\widehat{\Ccal}_K$ resp. $\widehat{\Rcal}$ for the restrictions of the stacks $\Ccal_K$ (resp. $\Rcal$) to the category ${\rm Nil}_p$
of $\Z_p$-schemes on which $p$ is locally nilpotent. See also \cite[§2]{phimod} for the definitions.\\
The motivations for these definitions are the following equivalences of categories (see \cite{Fontaine} and \cite[§1]{Kisin}).
\begin{prop}\label{phimodul}
Let  $A$ be a local Artin ring with residue field $\Fbb$ a finite field of characteristic $p$.  Then the category of $G_{K_\infty}$-representations on free $A$-modules of rank $d$ is equivalent to the category of \'etale $\phi$-modules over $(A\otimes_{\Z_p}W)((u))$ that are free of rank $d$.
\end{prop}
\begin{theo}\label{grpschemeclass}
Let $p>2$. Then there is an equivalence between the groupoid of finite flat group schemes $\Gcal$ over $\Spec\ \Ocal_K$ and the groupoid of pairs $(\Mfrak,\Phi)$, where $\Mfrak$ is a $W[[u]]$-module of projective dimension $1$ and $\Phi:\Mfrak\rightarrow \Mfrak$ is a $\phi$-linear map such that the cokernel of the linearisation of $\Phi$ is killed by $E(u)$. Under this equivalence the restriction of the Tate twist of the $G_K$-representation on $\Gcal(\bar K)$ to $G_{K_\infty}$ corresponds to the \'etale $\phi$-module $(\Mfrak[1/u],\Phi)$.   
\end{theo}
Further we will use the following notations: Let $(A,\mfrak)$ be a complete noetherian $W(\Fbb)$-algebra and $\xi:\Spf\ A\rightarrow \widehat{\Rcal}$ be an $A$-valued point of $\widehat{\Rcal}$. Write $\xi_n$ for the reduction of $\xi$ modulo $\mfrak^{n+1}$. By \cite[Corollary 2.6; 3.b]{phimod} the fiber product
\[\Spec(A/\mfrak^{n+1})\times_\Rcal\Ccal_K\]
is representable by a projective $A/\mfrak^{n+1}$-scheme $C_K(\xi_n)$ that is a closed subscheme of some affine Grassmannian over $\Spec(A/\mfrak^{n+1})$ for all $n\geq 0$. These schemes give rise to a formal scheme $\widehat{C}_K(\xi)$ over $\Spf\ A$. Using the very ample line bundle on the affine Grassmannian this formal scheme is algebraizable. The resulting projective scheme over $\Spec\ A$ will be denoted by $C_K(\xi)$.
\begin{rem}\label{charpok}
Note that this does not give an arrow $C_K(\xi)\rightarrow \Ccal_K$. For example the module $\Mfrak=W[[u]]$ together with the $\phi$-linear map $\Phi$ given by $\Phi(1)=E(u)$ does not define a $\Z_p$-valued point of $\Ccal_K$ but rather a "formal" point 
\[\Spf\ \Z_p\rightarrow \widehat{\Ccal}_K.\]
However if $B$ is some $\Z_p$-algebra killed by some power of $p$, then 
\[E(u)\in (B\widehat{\otimes}_{\Z_p}W((u)))^\times,\] and hence any locally free 
$B\widehat{\otimes}_{\Z_p}W[[u]]$-modules $\Mfrak$ with semi-linear map $\Phi$ satisfying 
\[E(u)\Mfrak\subset \Phi(\phi^\ast\Mfrak)\subset \Mfrak\]
defines a $B$-valued point of $\Ccal_K$.
\end{rem}

\section{The image of the coefficient space}

In the following we will assume that the representation $\bar\rho$ is flat (i.e. is the generic fiber of some finite flat group scheme over $\Spec\ \Ocal_K$) and that $\Dcal_{\bar\rho}^\fl$ is representable. This is the case if, for example, $\Endo_{\Fbb}(\bar\rho)=\Fbb$ (cf. \cite[Theorem 2.3]{Conrad}).
We write $\rho$ for the universal flat deformation.\\
By Proposition $\ref{phimodul}$ we have a map
\begin{equation}\label{deformequiv}
\Dcal_{\bar\rho_\infty}\longrightarrow \widehat{\Rcal},\end{equation}
see also \cite[4.a]{phimod} and \cite[1.2.6, 1.2.7]{Kisin}.
For some local Artinian ring $A$ and some $\xi\in\Dcal_{\bar\rho_\infty}(A)$ we write $M(\xi)\in\widehat{\Rcal}(A)$ for the corresponding $\Phi$-module. More precisely, this map identifies $\Dcal_{\bar\rho_\infty}$ with $\widehat{\Rcal}_{[\bar\rho_\infty]}$ (cf. \cite[4.a]{phimod}).
The latter functor is given by all deformations in $\widehat{\Rcal}$ of the $\Phi$-module $M(\bar\rho_\infty)$. Especially we find that the map in $(\ref{deformequiv})$ is formally smooth.
\begin{lem}\label{lemcartesian}
The restriction of $\rho$ to $G_{K_\infty}$ induces a map $\Spf\ R^\fl\rightarrow \Dcal_{\bar\rho_\infty}$. Composing the canonical projection $\widehat{C}_K(\rho)\rightarrow \Spf\ R^\fl$ with this morphism we obtain a $2$-cartesian diagram of stacks on local Artinian $W(\Fbb)$-algebras:
\[
\begin{xy}
\xymatrix{
\widehat{C}_K(\rho) \ar[r]\ar[d]  & \widehat{\Ccal}_K  \ar[d]\\
\Dcal_{\bar\rho_\infty} \ar[r] & \widehat{\Rcal}.
}
\end{xy}
\] 
\end{lem}
\begin{proof}
Let $A$ be a local Artinian $W(\Fbb)$-algebra such that $p^nA=0$ for some $n>0$. We have to show that there is a natural equivalence of categories
\[\widehat{C}_K(\rho)(A)\rightarrow \left\{
{\begin{array}{*{20}c}
((\Mfrak,\Phi),\xi,\alpha)\ \text{with}\  (\Mfrak,\Phi)\in\Ccal_K(A)\ ,\ \xi\in\Dcal_{\bar\rho_\infty}(A)\\
\text{and an isomorphism}\ \alpha:(\Mfrak[\tfrac{1}{u}],\Phi)\longrightarrow M(\xi)
\end{array}}\right\}.
\]
First it is clear that $\Spf\ R\rightarrow \Dcal_{\bar\rho_\infty}$ induces a natural map from the left to the right which is fully faithful. We have to show that it is essentially surjective.\\
Let $x=((\Mfrak,\Phi),\xi,\alpha)$ be an $A$-valued point of the right hand side. \\Then $(\Mfrak,\Phi)\in\Ccal_K(A)$ and by \cite[Proposition 4.3]{phimod} there is an associated flat representation $\widetilde{\xi}$ of $G_K$ such that 
\[\beta:(\Mfrak[1/u],\Phi)\longrightarrow\hspace{-5mm}^{=}\hspace{3mm}M(\widetilde{\xi}|_{G_{K_\infty}}).
\]
This shows that $y=((\Mfrak,\Phi),\widetilde{\xi},\beta)$ defines a unique point in $\widehat{C}_K(\rho)(A)$. 
It follows from the construction that this point maps to $x$. 
\end{proof}
\begin{rem}
Note that it is not clear whether $\Dcal_{\bar\rho_\infty}$ is representable, even if $\bar\rho_\infty$ is absolutely irreducible, since $G_{K_\infty}$ does not satisfy Mazur's $p$-finiteness condition (cf. \cite[§1. Definition]{Mazur}). As there is an isomorphism
\[G_{K_\infty}=\Gal(\bar K/K_\infty)\cong \Gal(k((u))^{\rm sep}/k((u))),\]
each open subgroup of finite index $H\subset G_{K_\infty}$ is isomorphic to the absolute Galois group of some local field in characteristic $p$,
\[H\cong \Gal(l((t))^{\rm sep}/l((t))),\]
where $l$ is a finite extension of $k$ and $t$ is an indeterminate. Hence by Artin-Schreier theory (cf. \cite[X §3.a]{Serre} for example),
there is an isomorphism 
\[\Hom_{\rm cont}(H,\Z/p\Z)\cong l((t))/\wp(l((t))),\]
and the latter group is infinite.\\
If one restricts the attention to $G_{K_\infty}$-representations of $E$-height $\leq h$ , then the deformation functor $\Dcal_{\bar\rho_\infty}^{\leq h}$ is representable if $\Endo_\Fbb(\bar\rho_\infty)=\Fbb$ (see \cite[Theorem 11.1.2]{Kim}). The $E$-height of a $p$-torsion $G_{K_\infty}$-representation is defined as the minimal $h$ such that the \'etale $\phi$-modules associated to the representation admits an $W[[u]]$-lattice with cokernel of the linearisation of $\Phi$ killed by $E(u)^h$ (see \cite[Definition 5.2.8]{Kim} for the precise definition).
\end{rem}
\begin{prop}\label{localstructure}
Let $C_K(\rho)$ denote the projective $R^\fl$-scheme obtained from $\widehat{C}_K(\rho)$ by algebraization.
Then $C_K(\rho)\otimes_{W(\Fbb)}W(\Fbb)[1/p]$ is reduced, normal and Cohen-Macaulay. The reduced subscheme underlying the special fiber $C_K(\rho)\otimes_{W(\Fbb)}\Fbb$ is normal and with at most rational singularities. Further the scheme $C_K(\rho)$ is topologically flat, i.e. its generic fiber is dense.
\end{prop}
\begin{proof}
This is similar to \cite[Proposition 2.4.6]{Kisin}.\\
Denote by $y:\Spec\ \Fbb\rightarrow \Rcal$ the $\Fbb$-valued point defined by $\bar\rho_\infty$. Let $x$ be a closed point of $C_K(\rho)$.
Extending scalars if necessary, we may assume that $x$ is defined over $\Fbb$. Denote by $(M_0,\Phi_0)\in\Rcal(\Fbb)$ the $\Phi$-module defined by $y$
and by $(\Mfrak_0,\Phi_0)\in\Ccal_K(\Fbb)$ the $\Phi$-module defined by $x$. We want to compare the structure of the local ring $\Ocal_{C_K(\rho),x}$ (resp. its completion) to the structure of a local model $M_K$ defined in \cite[3.a]{phimod}. By loc. cit. Theorem 0.1. there is a "local model"-diagram
\begin{equation}\label{localmodel}
\begin{aligned}
\begin{xy}\xymatrix{
&\widetilde{\Ccal_K} \ar[ld]_\pi\ar[rd]^\phi & \\
\Ccal_K & & M_K,
}\end{xy}
\end{aligned}
\end{equation}
with $\pi$ and $\phi$ formally smooth. Here the $B$-valued points of the stack $\widetilde{\Ccal_K}$ are the $\Phi$-modules $(\Mfrak,\Phi)\in\Ccal_K(B)$ together with an isomorphism $\Mfrak\rightarrow (B\widehat{\otimes}_{\Z_p}W[[u]])^d$, for a $\Z_p$-algebra $B$.\\
We consider the following groupoids on local Artinian $W(\Fbb)$-algebras: Denote by $\Dcal_x$ and $\Dcal_y$ the groupoids
\begin{align*}
\Dcal_x(B)&=\left\{
{\begin{array}{*{20}c}
(\Mfrak,\Phi)\in\Ccal_K(B)\ \text{such that}\\
(\Mfrak\otimes_B (B/\mfrak_B),\Phi\otimes\id)\cong (\Mfrak_0\otimes_\Fbb(B/\mfrak_B),\Phi_0\otimes\id)
\end{array}}\right\},\\
\Dcal_y(B)&=
\left\{
{\begin{array}{*{20}c}
(M,\Phi)\in\Rcal(B)\ \text{such that}\\
(M\otimes_B (B/\mfrak_B),\Phi\otimes\id)\cong (M_0\otimes_\Fbb(B/\mfrak_B),\Phi_0\otimes\id)
\end{array}}\right\}.
\end{align*}
Fixing a basis of $\Mfrak_0$ we may view $x$ as an $\Fbb$-valued point of $\widetilde{\Ccal_K}$. Denote by $\widetilde{\Dcal_x}$ the groupoid of deformations of $x$ in $\widetilde{\Ccal_K}$.\\
Under the morphism $\phi$ in $(\ref{localmodel})$, the point $x$ maps to a point $\bar x$ of $M_K$. This point defines an $\Fbb\otimes_{\Z_p}\Ocal_K$-submodule $L\subset (\Fbb\otimes_{\Z_p}\Ocal_K)^d$. Let $\Dcal_{\bar x}$ be the groupoid of deformations of $\bar x$, i.e.
\[\Dcal_{\bar x}=\{B\otimes_{\Z_p}\Ocal_K-\text{submodules}\ \Lcal\subset(B\otimes_{\Z_p}\Ocal_K)^d\mid \Lcal\otimes_B (B/\mfrak_B)\cong L\otimes_\Fbb(B/\mfrak_B)\}.\]
This groupoid is pro-represented by the completion of the local ring $\Ocal_{M_K,\bar x}$. Now we have the following commutative diagram.
\[
\begin{xy}\xymatrix{
&&& \widetilde{\Dcal_x}\ar[ld]_\pi\ar[rd]^\phi & \\
\Spf\ \widehat{\Ocal}_{C_K(\rho),x} \ar[d]\ar[rr]^{\xi'}&& \Dcal_x\ar[d] && \Dcal_{\bar x}\\
\Dcal_{\bar\rho_\infty} \ar[rr]^{\xi}&& \Dcal_y,
}\end{xy}
\]
where the lower left square is cartesian by Lemma $\ref{lemcartesian}$. As remarked above $\xi$ is formally smooth and hence so is $\xi'$.  
As $\Dcal_{\bar x}$ is pro-represented by the complete local ring at some closed point of the local model $M_K$ and as $\xi$, $\pi$ and $\phi$ are formally smooth, the assertion of the Proposition is true if it is true for $M_K$. But if follows from the definitions (using the notation of \cite[3.c]{phimod}) that
\[M_K\otimes_{\Z_p}\Q_p=\coprod_{\mu_i}M_{\mu_i,K}^{\rm loc}\otimes_{\Z_p}\Q_p,\]
for some cocharacters
\[\mu_i:\mathbb{G}_{m,\bar\Q_p}\longrightarrow (\Res_{K/\Q_p}\GL_d)_{\bar\Q_p}=\prod_{\psi:K\rightarrow \bar\Q_p}\GL_{d,\bar\Q_p},\]
such that
\[\bar\Q_p^d=\bigoplus_{n\in\{0,1\}}V_{n,i}^\psi,\]
where $V_{n,i}^\psi=\{v\in\bar\Q_p^d\mid(\pr_\psi\circ\mu_i)(a)v=a^nv\ \text{for all}\ a\in\bar\Q_p^\times\}$ and each of the $M_{\mu_i,K}^{\rm loc}$ is a local model in the sense of \cite{lokmod} (compare \cite[Remark 3.3]{phimod}). Hence, by \cite[Theorem 5.4]{lokmod}, the generic fiber of the local model $M_K\otimes_{\Z_p}\Q_p$ is normal, reduced and Cohen-Macaulay. The special fiber decomposes as follows: 
\[M_K\otimes_{\Z_p}\Fbb_p=\coprod_{\nu}M_{\mu_{\max}(\nu),K}^{\rm loc}\otimes_{\Z_p}\Fbb_p,\]
where $\nu$ runs over all cocharacters 
\[\mathbb{G}_{m,\bar\Q_p}\longrightarrow (\Res_{K/\Q_p}\mathbb{G}_m)_{\bar\Q_p},\] 
and where $\mu_{\max}(\nu)$ is the maximal dominant cocharacter $\mathbb{G}_m\rightarrow \Res_{K/\Q_p}\GL_d$ (for the dominance order) such that the composition
\[\begin{xy}
\xymatrix{
\mathbb{G}_{m,\bar\Q_p}\ar[r]&(\Res_{K/\Q_p}\GL_d)_{\bar\Q_p}\ar[r]^{\det}&(\Res_{K/\Q_p}\mathbb{G}_m)_{\bar\Q_p}
}
\end{xy}\]
equals $\nu$.
Now the claim again follows from \cite[Theorem 5.4]{lokmod}.
\end{proof}
\begin{rem}
We need to formulate the result on the local structure of the special fiber as a result about the underlying reduced scheme as the local models $M_{\mu,K}^{\rm loc}$ are in general not defined over $\Z_p$ but over a ramified extension and hence there are nilpotent elements in the special fiber $M_K\otimes_{\Z_p}\Fbb_p$.
\end{rem}
\begin{prop}\label{charzeroiso}
The map $C_K(\rho)\rightarrow \Spec\ R^\fl$ becomes an isomorphism in the generic fiber over $W(\Fbb)$, i.e.
\[\begin{xy}\xymatrix{
C_K(\rho)\otimes_{W(\Fbb)}{\rm Frac}(W(\Fbb))\ar[r]^{\hspace{10mm}\cong} & \Spec(R^\fl[\tfrac{1}{p}]).}
\end{xy}\]
\end{prop}
\begin{proof}
Using the result on the local structure of $C_K(\rho)$, the proof is the same as in \cite[Proposition 2.4.8]{Kisin}. The main point is to check that the map is a bijection on points.
\end{proof}
\begin{theo}\label{maintheo}
Suppose that the universal flat deformation ring $R^{\rm fl}$ of $\bar\rho$ exists and denote by $\rho$ the universal flat deformation.
Then the morphism $C_K(\rho)\rightarrow \Spec\ R^\fl$ is topologically surjective.
\end{theo}
We will prove this theorem in section $4$ below. We will conclude this section with some consequences of Theorem $\ref{maintheo}$.
\begin{cor}
Assume that $R^\fl$ exists, then $\Spec\,R^\fl$ is topologically flat.
\end{cor}
\begin{proof}
This follows from Theorem \ref{maintheo} and the corresponding result on $C_K(\rho)$.
\end{proof}
\begin{prop}
Assume that the universal deformation ring $R$ of $\bar\rho$ exists with universal deformation $\rho^{\rm univ}$. Then $C_K(\rho^{\rm univ})\rightarrow \Spec\,R$ factors over $\Spec\,R^{\fl}$ and is (canonically) isomorphic to $C_K(\rho)$.
\end{prop}
\begin{proof}
For $n\geq 0$ denote by $\rho_n:\Spec(R^\fl/\mfrak_{R^\fl}^{n+1})\rightarrow \Rcal$ the reduction of $\rho$ modulo $\mfrak_{R^\fl}^{n+1}$, and similarly $(\rho^{\rm univ})_n$. We consider the following diagram with all rectangles cartesian.
\[\begin{xy}\xymatrix{
C_K(\rho_n)\ar[rr]\ar[d] && C_K(\rho^{\rm univ}_n) \ar[rr]\ar[d] && \Ccal_K\ar[d]\\
\Spec(R^\fl/\mfrak^{n+1}_{R^\fl})\ar[rr]\ar@/_1pc/[rrrr]_{\rho_n} && \Spec(R/\mfrak^{n+1}_R)\ar[rr]^{\hspace{5mm}(\rho^{\rm univ})_n} && \Rcal.
}\end{xy}\]
By \cite[Proposition 4.3]{phimod} the morphism $C_K(\rho^{\rm univ}_n)\rightarrow \Spec(R/\mfrak^{n+1}_R)$ factors over $\Spec(R^\fl/\mfrak^{n+1}_{R^\fl})$ and hence $\widehat{C}_K(\rho_n)\rightarrow \widehat{C}_K(\rho_n^{\rm univ})$ is an isomorphism. \\
As $\Spf\,R^\fl\rightarrow \Spf\,R$ is a closed immersion the formal scheme $\widehat{C}_K(\rho)$ is a projective formal $\Spf\,R$-scheme and applying formal GAGA (see \cite[5.4]{EGA3}) over $\Spf\,R$ we find that also the algebraizations $C_K(\rho)$ and $C_K(\rho^{\rm univ})$ are isomorphic over $\Spec\,R$.
\end{proof}
\begin{prop}\label{lowram}
Assume that $e=[K:K_0]<p-1$. Then the morphism 
\[C_K(\rho)\rightarrow \Spec\ R^\fl\]
is an isomorphism.
\end{prop}
\begin{proof}
It is enough to show that $\widehat{C}_K(\rho)\rightarrow \Spf\ R^\fl$ is an isomorphism. We show that both objects pro-represent the same functor, i.e. $\widehat{C}_K(\rho)$ pro-represents the deformation functor $\Dcal_{\bar\rho}^\fl$. Let $A$ be a local Artinian ring and $\xi\in\Dcal^\fl_{\bar\rho}(A)$ a flat deformation of $\bar\rho$. By a result of Raynaud (cf. \cite[Proposition 3.3.2]{Raynaud}) there is a unique flat model for this deformation.
Denote by $(\Mfrak,\Phi)$ the $W[[u]]$-module associated with this group scheme by Kisin's classification. This is a $W[[u]]$-submodule of the \'etale $\phi$-module $(M,\Phi)$ over $A_W((u))$ corresponding to the (twist of) the restriction of $\xi$ to $G_{K_\infty}$. Replacing $\Mfrak$ by its $A_W[[u]]$-span inside $M$, we may assume that it is an $A_W[[u]]$-submodule of $M$. Applying the argument of \cite[Remark 4.4]{phimod} we find that $\Mfrak$ is free over $A_W[[u]]$
This defines the unique point in $\Ccal_K(A)$ above $\xi$. We have shown that the functor morphism $\widehat{C}_K(\rho)\rightarrow \Dcal^\fl_{\bar\rho}$ is bijective on $A$-valued points. The claim follows.
\end{proof}
\begin{rem}
All the above results also apply to framed deformation rings. We need to replace the deformation functors by deformation groupoids and the fiber products by $2$-fiber products.
 For the corresponding result on the local structure one only needs that the morphism $\Dcal^\square_{\bar\rho_\infty}\rightarrow \widehat{\Rcal}$ is smooth, where $\Dcal^\square_{\bar\rho_\infty}$ denotes the groupoid of framed deformations of $\bar\rho_\infty$.
The result for framed deformations (or for deformation stacks) can be stated as follows:
Given a field $F$ of characteristic $p$ and a morphism $\Spec\,F\rightarrow \Spec\,R^\fl$. There exists an fpqc-cover $F'$ of $F$ and a integral complete local ring $(A,\mfrak)$ with ${\rm char}({\rm Frac}\,A)=0$ and $A/\mfrak=F'$ such that the composition
\[\Spec\,F'\longrightarrow \Spec\,F\longrightarrow\Spec\,R^\fl\]
lifts to a morphism $\Spec\,A\rightarrow \Spec\,R^\fl$. 
\end{rem}
\begin{rem}
If the prime $p$ equals $2$, then there is a similar classification of finite flat group schemes as in Theorem $\ref{grpschemeclass}$, but it only applies to connected group schemes (see \cite{Kisinp=2}). Hence the same results hold in the case $p=2$, if one considers deformations that are the generic fiber of a connected finite flat group scheme. 
\end{rem}
\section{Proof of Theorem \ref{maintheo}}
In this section we prove the main result, Theorem $\ref{maintheo}$. \\
Let $e=[K:K_0]$ denote the ramification index of $K$ over $\Q_p$. Then the degree of the Eisenstein polynomial $E(u)$ is $e$ and its reduction modulo $p$ is $u^e\in k[u]$.\\
For the rest of the section we denote by $\Ocal_F=l[[\varpi]]$ a complete discrete valuation ring in characteristic $p$
with finite residue field $l$ containing $k$. We will use the notation $A_n=\Ocal_F/(\varpi^{n+1})\otimes_{\Fbb_p}k$. For a ring $R$ and a free $R((u))$-module $R((u))^d$, a finitely generated projective $R[[u]]$-submodule that generates $R((u))^d$ will be called a lattice in $R((u))^d$. Finally, we will write $\Ocal_F\{\{u\}\}$ for the $\varpi$-adic completion of $\Ocal_F((u))$.
\begin{lem}\label{boundedtorsion}
Let $(M,\Phi)\in\Rcal(\Ocal_F/(\varpi^{n+1}))$ and $\Mfrak\subset M$ a finitely generated $A_n[[u]]$-submodule such that $\Mfrak[1/u]=M$ and 
\[u^e\Mfrak\subset \Phi(\phi^\ast\Mfrak)\subset\Mfrak.\]
Then the $l$-dimension of the $u$-torsion part of the finitely generated $l[[u]]$-module $\Mfrak/\varpi\Mfrak$ is bounded by
\[\dim_l (\Mfrak/\varpi\Mfrak)^{\rm tors}\leq [k:\Fbb_p]d\tfrac{e}{p-1}.\]
\end{lem} 
\begin{proof}
We can describe the $u$-torsion as follows.
\[\Mfrak/\varpi\Mfrak=\bigoplus_{i=0}^n (\Mfrak\cap\varpi^iM)/(\varpi\Mfrak\cap\varpi^{i}M +\Mfrak\cap\varpi^{i+1}M).\]
In this direct sum the summand for $i=0$ is the free part in the quotient and the $i$-th summand is the contribution of the 
elements in $\Mfrak\cap(\varpi^iM\backslash\varpi^{i+1}M)$ to the $u$-torsion. Further for $i\in{1,\dots,n-1}$ we have
\begin{align*}
&\dim_l (\Mfrak\cap \varpi^iM)/(\varpi^i\Mfrak+\Mfrak\cap\varpi^{i+1}M)\\+&\dim_l (\Mfrak\cap\varpi^{i+1}M)/(\varpi\Mfrak\cap\varpi^{i+1}M +\Mfrak\cap\varpi^{i+2}M)\\
=&\dim_l (\Mfrak\cap \varpi^{i+1}M)/(\varpi^{i+1}\Mfrak+\Mfrak\cap\varpi^{i+2}M).\end{align*}
This can be seen using the interpretation of $\dim_l (\Mfrak\cap \varpi^iM)/(\varpi^i\Mfrak+\Mfrak\cap\varpi^{i+1}M)$ as the sum of all elementary divisors
of the lattice $(\Mfrak\cap\varpi^iM)/(\Mfrak\cap\varpi^{i+1}M)$ with respect to $\varpi^i\Mfrak/(\Mfrak\cap\varpi^{i+1}M)$ as $l[[u]]$-lattices in $\varpi^iM/\varpi^{i+1}M$ and the fact that the multiplication by $\varpi$ induces isomorphisms from $\varpi^iM/\varpi^{i+1}M$ to $\varpi^{i+1}M/\varpi^{i+2}M$ for $i\leq n-1$.
Now we find that
\[\dim_l(\Mfrak/\varpi\Mfrak)^{\rm tors}=\dim_l (\Mfrak\cap\varpi^nM)/\varpi^n\Mfrak.\]
The Lemma now follows from the following claim: 
\[u^{\lfloor\tfrac{e}{p-1}\rfloor}(\Mfrak\cap\varpi^nM)\subset \varpi^n\Mfrak.\]
We denote by $j$ the minimal integer such that $u^{j}(\Mfrak\cap\varpi^nM)\subset \varpi^n\Mfrak$.\\
Then $pj$ is the minimal integer $r$ such that $u^r\Phi(\phi^\ast (\Mfrak\cap\varpi^nM))\subset \varpi^n\Phi(\phi^\ast\Mfrak)$. But we have
\[\varpi^n\Phi(\phi^\ast\Mfrak)\supset u^e\varpi^n\Mfrak\supset u^{e+j}(\Mfrak\cap\varpi^nM)\supset u^{e+j}\Phi(\phi^\ast (\Mfrak\cap\varpi^nM)).\]
Hence $pj\leq e+j$ and the claim follows.
\end{proof}
\begin{lem}\label{finiteset}
Let $(M,\Phi)\in\Rcal(\Ocal_F/(\varpi^{n+1}))$. Then there are at most finitely many finitely generated $A_n[[u]]$-submodules $\Mfrak\subset M$ such that 
$\Mfrak[1/u]=M$ and 
\[u^e\Mfrak\subset\Phi(\phi^\ast\Mfrak)\subset \Mfrak.\]
\end{lem} 
\begin{proof}
The module $M$ is a $nd[k:\Fbb_p]$-dimensional $l((u))$ vector space. Every finitely generated $A_n[[u]]$ submodule $\Mfrak\subset M$ with $\Mfrak[1/u]=M$ is an $l[[u]]$-lattice in $M$. Hence the argument of \cite[Proposition 2.1.7]{Kisin} shows that there exists a lattice $\Mfrak_0\subset M$ and integers $i_1,i_2\in\Z$ such that all $\Mfrak\subset M$ satisfying the properties of the Lemma satisfy
\[u^{i_1}\Mfrak_0\subset\Mfrak\subset u^{i_2}\Mfrak_0.\]
These are only finitely many lattices.
\end{proof}
\begin{prop}\label{algebraizable}
Let $(\widehat{M},\widehat{\Phi})$ be a point $\Spf\ \Ocal_F\rightarrow \widehat{\Rcal}$ and denote by $(M_n,\Phi_n)$ the reduction modulo $\varpi^{n+1}$, i.e. the $\phi$-module defined by $\Spec\ \Ocal_F/\varpi^{n+1}\rightarrow \Rcal$. Assume that there exist finitely generated $A_n[[u]]$ submodules $\Mfrak_n\subset M_n$ such that
$\Mfrak_n[1/u]=M_n$ and 
\[u^e\Mfrak_n\subset\Phi_n(\phi^\ast\Mfrak_n)\subset\Mfrak_n.\]
Then there exists $(M,\Phi)\in\Rcal(\Ocal_F)$ such that
\[(M/\varpi^{n+1}M,\Phi\mod\varpi^{n+1})=(M_n,\Phi_n).\]
\end{prop}
\begin{proof}
We denote by $\mathcal{Z}_n$ the set of all finitely generated $A_n[[u]]$-submodules $\mathfrak{N}\subset M_n$ such that $\mathfrak{N}[1/u]=M_n$ and 
\[u^e\mathfrak{N}\subset\Phi_n(\phi^\ast\mathfrak{N})\subset \mathfrak{N}.\]
By assumption these sets are non empty and by Lemma $\ref{finiteset}$ they are finite. Further if $\mathfrak{N}\in M_n$ and $m<n$, then the image of $\mathfrak{N}$ under the map 
\[M_n\longrightarrow M_m\]
defines an element of $\mathcal{Z}_m$, denoted by $f_{nm}(\mathfrak{N})$. As the sets $\mathcal{Z}_i$ are non empty and finite we can inductively construct
a sequence $\widetilde{\Mfrak}_n\in\mathcal{Z}_n$ such that $f_{nm}(\widetilde{\Mfrak}_n)=\widetilde{\Mfrak}_m$ for $m\leq n$. We denote this sequence again by $\Mfrak_n$ instead of $\widetilde{\Mfrak}_n$.\\
By Lemma $\ref{boundedtorsion}$ there are only finitely many possibilities for the isomorphism class of the $u$-torsion in $\Mfrak_n/\varpi\Mfrak_n$.
Hence there exists a strictly increasing sequence $n_i\in\mathbb{N}$ such that 
\begin{equation}\label{equalutors}
\begin{xy}
\xymatrix{
\Mfrak_{n_i}/\varpi\Mfrak_{n_i}\ar[r]^{\cong}&\Mfrak_{n_j}/\varpi\Mfrak_{n_j}
}
\end{xy}
\end{equation}
for $i,j\in\mathbb{N}$. Now there is an isomorphism
\[\Mfrak_{n_i}/\varpi^{n_j+1}\Mfrak_{n_i}\cong \Mfrak_{n_j}\oplus (\Mfrak_{n_i}/\varpi^{n_j+1}\Mfrak_{n_i})^{\rm tors},\]
where the last summand is the $u$-torsion part of the left hand side. We find that
\[(\Mfrak_{n_i}/\varpi\Mfrak_{n_i})^{\rm tors}=(\Mfrak_{n_j}/\varpi\Mfrak_{n_j})^{\rm tors}\oplus ((\Mfrak_{n_i}/\varpi^{n_j+1}\Mfrak_{n_i})^{\rm tors})/\varpi.\]
Using $(\ref{equalutors})$ and Nakayama's Lemma it follows that
\[\begin{xy}\xymatrix{\Mfrak_{n_i}/\varpi^{n_j+1}\Mfrak_{n_i}\ar[r]^{\hspace{0.5cm}\cong}&\Mfrak_{n_j}}\end{xy}\]
for all $j\leq i$. 
Especially there is an $r\in\mathbb{N}$ (independent of $i$) and generators $b_1^{(i)},\dots,b_r^{(i)}$ of $\Mfrak_{n_i}$ as an $A_{n_i}[[u]]$-module
such that the $b^{(i)}_\nu$ reduce to $b^{(j)}_\nu$ modulo $\varpi^{n_j+1}$ for all $j\leq i$.
Choosing a compatible expression of $\Phi_n(b^{(i)}_j)$ in terms of the $b^{(i)}_j$ we can define commutative diagrams for $j\leq i$
\[\begin{xy}\xymatrix{
(A_{n_i}[[u]]^r,\widetilde{\Phi}_{n_i}) \ar[r]\ar[d] & (\Mfrak_{n_i},\Phi_{n_i})\ar[d] \\
(A_{n_j}[[u]]^r,\widetilde{\Phi}_{n_j}) \ar[r] & (\Mfrak_{n_i},\Phi_{n_j})
}\end{xy}\]
where all arrows are surjective. In the limit we get morphisms
\[\begin{xy}\xymatrix{
((\Ocal_F\otimes_{\Fbb_p}k)[[u]]^r,\widetilde{\Phi}) \ar[d]\\
\widetilde{M}=\lim\limits_{\leftarrow}(A_{n_i}((u))^r,\widetilde{\Phi}_{n_i}) \ar[r] & (\widehat{M},\widehat{\Phi}).
}\end{xy}\]
Here the lower arrow is surjective by the Mittag-Leffler criterion: The modules in question have finite length. Note that we do not claim that the linearisation of $\widetilde{\Phi}$ is an isomorphism after inverting $u$. \\
Now the image of the vertical arrow defines (after inverting $u$) a free $(\Ocal_F\otimes_{\Fbb}k)((u))$-submodule $\widetilde{N}$ of the $(\Ocal_F\otimes_{\Fbb_p}k)\{\{u\}\}$-module $\widetilde{M}$ such that
\[\begin{xy}\xymatrix{\widetilde{N}\otimes_{\Ocal_F((u))}\Ocal_F\{\{u\}\}\ar[r]^{\hspace{13mm}\cong} & \widetilde{M}.}\end{xy}\]
The image of $\widetilde{N}$ under $\widetilde{M}\rightarrow \widehat{M}$ defines a finitely generated $(\Ocal_F\otimes_{\Fbb_p}k)((u))$ submodule $N$ such that 
\[\begin{xy}\xymatrix{
N\otimes_{\Ocal_F((u))}\Ocal_F\{\{u\}\}\ar[r]^{\hspace{13mm}\cong} & \widehat{M}.
}\end{xy}\]
Further $N$ is $\widehat{\Phi}$-stable by construction. We claim that $N$ is free.\\
As $k\subset l$ we have isomorphisms 
\[(\Ocal_F\otimes_{\Fbb_p}k)((u))\longrightarrow\prod_{\psi:k\hookrightarrow l}\Ocal_F((u)).\]
And hence $\widehat{M}=\widehat{M}^{(\psi_0)}\times\widehat{M}^{(\phi\psi_0)}\times\dots\times\widehat{M}^{(\phi^{f-1}\psi_0)}$ where $\psi_0$ is a fixed embedding, $\phi$ is the absolute Frobenius on $k$ and $f=[k:\Fbb_p]$. 
The endomorphism $\widehat{\Phi}$ maps $\widehat{M}^{(\phi^i\psi_0)}$ to $\widehat{M}^{(\phi^{i+1}\psi_0)}$.
As $N$ is $\widehat{\Phi}$-stable we find that $N=N^{(\psi_0)}\times\dots\times N^{(\phi^{f-1}\psi_0)}$, where $N^{(\phi^i\psi_0)}$ is a finitely generated $\Ocal_F((u))$-submodule of $\widehat{M}^{(\phi^i\psi_0)}$ that generates $\widehat{M}^{(\phi^i\psi_0)}$ over $\Ocal_F\{\{u\}\}$ and hence is free of rank $d$, as $\Ocal_F((u))$ is principal.\\
Now $(N,\widehat{\Phi})$ is the object claimed in the Proposition: It follows from the construction that $(N,\widehat{\Phi})$ reduces to $(M_n,\Phi_n)$ modulo $\varpi^{n+1}$ and hence it follows from Nakayama's lemma that the linearisation of $\widehat{\Phi}$ is invertible on $N$.
\end{proof}
Before we continue we want to remind the reader that not every $A_n[[u]]$-submodule $\Mfrak_n\subset M_n$ satisfying $u^e\Mfrak_n\subset\Phi_n(\phi^\ast\Mfrak_n)\subset \Mfrak_n$ defines an $\Ocal_F/\varpi^{n+1}$-valued point of $\Ccal_K$. This is only the case if $\Mfrak_n$ is a free $A_n[[u]]$-module.
\begin{prop}\label{existlift}
Let $(M,\Phi)\in\Rcal(\Ocal_F)$ and denote by $(M_n,\Phi_n)\in\Rcal(\Ocal_F/\varpi^{n+1})$ the reduction modulo $\varpi^{n+1}$. Assume that there exist finitely generated $A_n[[u]]$-submodules $\Mfrak_n\subset M_n$ such that $\Mfrak_n[1/u]=M_n$ and 
\[u^e\Mfrak_n\subset \Phi_n(\phi^\ast\Mfrak_n)\subset\Mfrak_n.\]
Then there exists the diagonal arrow in the diagram
\[\begin{xy}\xymatrix{
&\Ccal_K \ar[d]\\
\Spec\ \Ocal_F \ar[ru]\ar[r] & \Rcal.
}\end{xy}\]
\end{prop}
\begin{proof}
Consider the free $(F\otimes_{\Fbb_p}k)((u))$-module $M\otimes_{\Ocal_F((u))}F((u))$.
We choose an $(F\otimes_{\Fbb_p}k)[[u]]$-lattice $\widetilde{\mathfrak{N}}\subset M\otimes_{\Ocal_F((u))}F((u))$. As the linearisation of $\Phi$ is an isomorphism, there exist $r\in\mathbb{N}$ such that
\[u^r\widetilde{\mathfrak{N}}\subset\Phi(\phi^\ast\widetilde{\mathfrak{N}})\subset u^{-r}\widetilde{\mathfrak{N}},\]
i.e. $\widetilde{\mathfrak{N}}$ is an $F$-valued point of the stack $\Ccal_r$ defined in $(\ref{stackCm})$.\\
By \cite[Corollary 2.6]{phimod} and the valuative criterion of properness, the diagonal arrow in the diagram below exists,
\[\begin{xy}\xymatrix{
\Spec\ F\ar[r]\ar[d]&\Ccal_r\ar[d] \\
\Spec\ \Ocal_F\ar[ru]\ar[r]&\Rcal.
}\end{xy}\]
This means that $\widetilde{\mathfrak{N}}$ extends to an $(\Ocal_F\otimes_{\Fbb_p}k)[[u]]$-lattice $\mathfrak{N}$ such that
\[u^r\mathfrak{N}\subset\Phi(\phi^\ast\mathfrak{N})\subset u^{-r}\mathfrak{N}.\]
We denote by $\mathfrak{N}_n$ the reduction of $\mathfrak{N}$ modulo $\varpi^{n+1}$.\\
By assumption there are finitely generated $A_n[[u]]$-submodules $\Mfrak_n\subset M_n$ such that $\Mfrak_n[1/u]=M_n$ and 
\[u^e\Mfrak_n\subset\Phi_n(\phi^\ast\Mfrak_n)\subset \Mfrak_n.\]
By the same argument as in the proof of Proposition $\ref{algebraizable}$ we can assume that $\Mfrak_n$ maps onto $\Mfrak_{n-1}$ under the projection $M_n\rightarrow M_{n-1}$ for all $n$. Now the argument of \cite[Proposition 2.1.7]{Kisin} shows that there is an integer $s$ only depending on $r$ and $e$ such that
\[u^s\mathfrak{N}_n\subset\Mfrak_n\subset u^{-s}\mathfrak{N}_n.\]
If we write $\Mfrak$ for $\lim\limits_{\leftarrow}\Mfrak_n$, then this shows
\[u^s\mathfrak{N}\subset\Mfrak\subset u^{-s}\mathfrak{N}.\]
Hence $\Mfrak$ is finitely generated over $\Ocal_F[[u]]$ and contains an $(\Ocal_F\otimes_{\Fbb_p}k)((u))$-basis of $M$. Further it still satisfies $u^e\Mfrak\subset\Phi(\phi^\ast\Mfrak)\subset\Mfrak$ and $\Mfrak\otimes_{\Ocal_F[[u]]}F[[u]]$ is free over $(F\otimes_{\Fbb_p}k)[[u]]$. Hence we obtain the following commutative diagram
\[
\begin{xy}
\xymatrix{
\Spec\ F\ar[r]\ar[d] & \Ccal_K^1\ar[d] \\
\Spec\ \Ocal_F\ar[r] & \Rcal,
}
\end{xy}
\]
where $\Ccal_K^1=\Ccal_K\times_{\Spec\ \Z_p}\Spec\ \Z/p\Z$ is the reduction of $\Ccal_K$ modulo $p$ (compare \cite[3.b.]{phimod}). By loc. cit. the stack $\Ccal_K^1$ is a closed substack of $\Ccal^1_e=\Ccal_e\times_{\Spec\ \Z_p}\Spec\ \Z/p\Z$. Using the valuative criterion of properness again we obtain the desired arrow.
\end{proof}
\begin{proof}[Proof of Theorem \ref{maintheo}]
By Proposition $\ref{charzeroiso}$ the morphism $C_K(\rho)\rightarrow \Spec\,R^\fl$ is an isomorphism in the generic fiber over $W(\Fbb)$. Especially it is surjective. \\
We write $\overline{C}_K(\rho)=C_K(\rho)\otimes_{W(\Fbb)}\Fbb$ for the special fiber of $C_K(\rho)$ and $\overline{R}^\fl$ for $R^\fl/pR^\fl$.
Let $\eta$ be a point of $\overline{R}^\fl$ that is not the unique closed point $x_0$. We mark the specialization $\eta\rightsquigarrow x_0$ by a morphism
\[\Spec\ \Ocal_F\longrightarrow \overline{R}^\fl.\]
where $\Spec\ \Ocal_F$ is a complete discrete valuation ring and the morphism maps the generic point of $\Spec\ \Ocal_F$ to $\eta$ and the special point to $x_0$. By a Zariski density argument is suffices to assume that the residue field of $\Ocal_F$ contains $k$ (recall that $R^\fl$ is a quotient of a power series ring in finitely many variables over $W(\Fbb)$).
The morphism
\begin{equation}\label{localmap}
\Spf\ \Ocal_F\longrightarrow \Spf\overline{R}^\fl\longrightarrow \Dcal_{\bar\rho_\infty}
\end{equation}
induces modules $(M_n,\Phi_n)\in\Rcal(\Ocal_F/\varpi^{n+1})$. By Kisin's classification of finite flat group schemes (Theorem $\ref{grpschemeclass}$) there exist finitely generated $k[[u]]$-submodules $\Mfrak_n\subset M_n$ such that $\Mfrak_n[1/u]=M_n$ and 
\[u^e\Mfrak_n\subset\Phi_n(\phi^\ast\Mfrak_n)\subset\Mfrak_n.\]
Replacing $\Mfrak_n$ by the $(\Ocal_F/\varpi^{n+1}\otimes_{\Fbb_p}k)[[u]]$-modules that it generates, we may assume that $\Mfrak_n$ is stable under the action of $\Ocal_F/\varpi^{n+1}$.
By Proposition $\ref{algebraizable}$ the arrow $\Spf\ \Ocal_F\rightarrow \widehat{\Rcal}$ is algebraizable to a morphism $\Spec\ \Ocal_F\rightarrow \Rcal$ and by Proposition $\ref{existlift}$ we obtain a commutative diagram
\[\begin{xy}\xymatrix{
& \Ccal_K \ar[d]\\
\Spec\ \Ocal_F \ar[ru]\ar[r]& \Rcal.
}\end{xy}\]
This yields a commutative diagram 
\[\begin{xy}\xymatrix{
& \widehat{C}_K(\rho)\ar[d]\\
\Spf\ \Ocal_F\ar[ru]\ar[r] & \Dcal_{\bar\rho_\infty}.
}\end{xy}\] 
We have to show that the arrow $\Spf\ \Ocal_F\rightarrow  \Dcal_{\bar\rho_\infty}$ factors over $R^\fl$ and that this morphism coincides with the arrow in $(\ref{localmap})$.\\
By \cite[Proposition 4.3]{phimod} we obtain from the morphisms $\Spec\ \Ocal_F/\varpi^{n+1}\rightarrow \Ccal_K$ flat $G_K$-representations such that the restriction to $G_{K_\infty}$ induces the objects $(M_n,\Phi_n)$ under the morphism $(\ref{deformequiv})$. 
By \cite[Theorem 3.4.3]{Breuil} the restriction to $G_{K_\infty}$ is fully faithful on the category of flat $p$-torsion $G_K$-representations and hence
the two morphisms $\Spf\ \Ocal_F\rightarrow \Spf\ R^\fl$ coincide. This yields the claim.
\end{proof}

\medskip

\address{Mathematisches Institut der Universit\"at Bonn, Endenicher Allee 60, 53115 Bonn, Germany}\\
\email{hellmann@math.uni-bonn.de}

\end{document}